\documentclass{article}%
\usepackage{amsmath}
\usepackage{amsfonts}
\usepackage{amssymb}
\usepackage{graphicx}%
\providecommand{\U}[1]{\protect\rule{.1in}{.1in}}
\newtheorem{theorem}{Theorem}

\newtheorem{example}[theorem]{Example}

\newtheorem{lemma}[theorem]{Lemma}

\newtheorem{remark}[theorem]{Remark}

\newenvironment{proof}[1][Proof]{\noindent\textbf{#1.} }{\ \rule{0.5em}{0.5em}}
\begin{document}

\title{Aspects of the Segre variety $\mathcal{S}_{1,1,1}(2)$}
\author{Ron Shaw, Neil Gordon, Hans Havlicek}
\maketitle

\begin{abstract}
We consider various aspects of the Segre variety $\mathcal{S}:=\mathcal{S}%
_{1,1,1}(2)$ in $\operatorname{PG}(7,2),$ whose stabilizer group
$\mathcal{G}_{\mathcal{S}}<\operatorname*{GL}(8,2)$ has the structure
$\mathcal{N}\rtimes\operatorname*{Sym}(3),$ where $\mathcal{N}%
:=\operatorname*{GL}(2,2)\times\operatorname*{GL}(2,2)\times\operatorname*{GL}%
(2,2).$ In particular we prove that $\mathcal{S}$ determines a distinguished
$Z_{3}$-subgroup $\mathcal{Z}<\operatorname*{GL}(8,2)$ such that
$A\mathcal{Z}A^{-1}=\mathcal{Z},$ for all $A\in\mathcal{G}_{\mathcal{S}},$ and
in consequence $\mathcal{S}$ determines a $\mathcal{G}_{\mathcal{S}}%
$-invariant spread of $85$ lines in $\operatorname*{PG}(7,2).$ Furthermore we
see that Segre varieties $\mathcal{S}_{1,1,1}(2)$ in $\operatorname{PG}(7,2)$
come along in triplets $\{\mathcal{S},\mathcal{S}^{\prime},\mathcal{S}%
^{\prime\prime}\}$ which share the same distinguished $Z_{3}$-subgroup
$\mathcal{Z}<\operatorname*{GL}(8,2).$ We conclude by determining all fifteen
$\mathcal{G}_{\mathcal{S}}$-invariant polynomial functions on
$\operatorname*{PG}(7,2)$ which have degree \thinspace$<8$, and their relation
to the five $\mathcal{G}_{\mathcal{S}}$-orbits of points in
$\operatorname*{PG}(7,2).$

\end{abstract}
\date{~}

\bigskip



\section{Introduction \label{Sec Intro}}

We work over $\operatorname*{GF}(2),$ and so \emph{we may identify the nonzero
elements of a vector space }$V(n+1,2)=V_{n+1}$ \emph{ with the points }$S_{0}$
\emph{of the associated projective space }$\mathbb{P}V_{n+1}=\operatorname{PG}%
(n,2).$ Consequently we identify $\operatorname*{GL}(V_{n+1}%
)=\operatorname*{GL}(n+1,2)$ with the group $\operatorname*{PGL}(n+1,2)$ of
collineations of $\operatorname*{PG}(n,2).$ We use $\prec u,v,\ldots\succ$ for
the vector subspace spanned by vectors $u,v,\ldots,$ and $\langle
u,v,\ldots\rangle$ for the flat (projective subspace) generated by projective
points $u,v,\ldots\;.$ The vector space $F(S_{0})$ of all functions
$S_{0}\rightarrow\operatorname*{GF}(2)$ is of dimension $|S_{0}|=2^{n+1}-1.$
Given a point-set $\psi\subset\operatorname{PG}(n,2)$ it has equation
$\tilde{Q}(x)=0$ for some polynomial $\tilde{Q}$ satisfying $\tilde{Q}(0)=0.$
Upon replacing $(x_{i})^{r_{i}},r_{i}>1,$ by $x_{i}$ in any such polynomial we
obtain a \emph{uniquely determined} polynomial $Q=Q_{\psi}$ of the form $\sum
x_{i_{1}}\cdots x_{i_{k}}$, $1\leq i_{1}<\cdots<i_{k}\leq n+1$. (This
uniqueness does \emph{not} hold for a point-set $\psi\subset\operatorname{PG}%
(n,q)$ for $q>2:$ see for example \cite[Remark 1.2]{ShawPsiAssociate}.)
Briefly stated, \emph{every point-set of }$\operatorname*{PG}(n,2)$\emph{ is a
hypersurface. }The (reduced) degree $d=\deg Q$ of $Q$ is the \emph{polynomial
degree }of the point-set $\psi.$ (On account of the afore-mentioned
uniqueness, the reduced degree $d$ of $Q$ is seen to be independent of the
coordinate system.) It should be noted, see \cite[Section 1.2]%
{ShawGordonPolDegGrass}, that \emph{if }$|\psi|$\emph{ is odd then }$d\leq
n$\emph{ while if }$|\psi|$\emph{ is even then }$d=n+1.$ Further note that if
$F_{d}=F_{d}(S_{0}),$ $d>0,$ denotes the subspace of $F(S_{0})$ which consists
of functions $f$ expressible as a polynomial function $f(x_{1},x_{2}%
,\ldots,x_{n+1})$ with $\deg f\leq d$ and $f(0)=0,\ $then the subspaces
$F_{d}$ are nested:%
\begin{equation}
F_{1}\subset F_{2}\subset\cdots\subset F_{n}\subset F_{n+1}=F(S_{0}).
\label{Nesting of F_r}%
\end{equation}

Given a choice of subset $\psi\subset\operatorname{PG}(n,2)$ a flat $X$ of
$\operatorname{PG}(n,2)$ is termed $\psi$-\emph{odd} whenever $|X\cap\psi|$ is
odd and $\psi$-\emph{even} whenever $|X\cap\psi|$ is even. The next lemma
shows that the degree $d=\deg Q$ can be determined from the point-set $\psi$
purely by incidence properties.

\begin{lemma}
\label{Lem pol deg d iff (i), (ii)}(See \cite[Theorem 1.1]{ShawBBC20}.) If
$|\psi|$ is odd then $Q$ has polynomial degree $d$ if and only if \emph{(i)
}every $d$-flat is $\psi$-odd and \emph{(ii)} there exists at least one
$(d-1)$-flat which is $\psi$-even. (Here condition (i) $\Longrightarrow\deg
Q\leq d,$ and condition (ii) $\Longrightarrow\deg Q\geq d.$)
\end{lemma}

When attempting to determine the polynomial degree of a hypersurface $\psi$ in
$\operatorname*{PG}(n,2)$ it often helps to make use of the following
elementary lemma.

\begin{lemma}
\label{Lem poldeg of flat X}Each $(n-d)$-flat $X$ in $\operatorname*{PG}(n,2)$
has polynomial degree $d.$ In detail, if $X$ is the intersection of the $d$
independent hyperplanes $f_{1}(x)=0,\ldots,f_{d}(x)=0,$ where the $f_{i}$ are
elements of the dual $V_{n+1}^{\ast}$ of $V_{n+1},$ then $X$ has equation
\begin{equation}
P_{X}(x)=0,\text{ \quad where }P_{X}:=1+\Pi_{i=1}^{d}(1+f_{i}), \label{Q_X}%
\end{equation}
the polynomial $P_{X}$ thus having degree $d.$
\end{lemma}

\begin{proof}
$1+\Pi_{i=1}^{d}(1+f_{i}(x))$ equals $1$ except when $f_{1}(x)=\cdots
=f_{d}(x)=0.$
\end{proof}

\medskip

In fact, from now onwards, we confine our attention to projective dimension
$n=7,$ and moreover will deal solely with the Segre variety $\mathcal{S}%
_{1,1,1}(2)\subset\operatorname*{PG}(7,2),$ see \cite{Burau} or
\cite{HirschfeldThas3}, and with its stabilizer $\mathcal{G}_{\mathcal{S}%
}:=\mathcal{G}_{\mathcal{S}_{1,1,1}(2)}<\operatorname*{GL}(V_{8}).$ We will be
particularly concerned with various $\mathcal{G}_{\mathcal{S}}$-invariant
attributes of the Segre variety and so will overlap at times with some of the
material in the recent paper \cite{HavlicekOdehnalSaniga}. However we will
work entirely over the field $\operatorname*{GF}(2),$ in contrast to the
frequent excursions into $\operatorname*{GF}(4)$ territory undergone in
\cite{HavlicekOdehnalSaniga}. There will also be neat connections to
\cite{GreenSaniga}. The $3\times3\times3$ grid and its Veldkamp space from
\cite{GreenSaniga} are in our setting the Segre variety $\mathcal{S}%
_{1,1,1}(2)$ (considered merely as a point-line geometry) and the dual of the
ambient $\operatorname*{PG}(7,2)$ (recovered in terms of this point-line
geometry). It is worth noting that the authors of \cite{GreenSaniga} took
their motivation from physics. They were looking for finite geometries
potentially allowing physical applications, similar to the ones in
\cite{HavlicekOdehnalSaniga1}, \cite{LevaySanigaVrana} and
\cite{LevaySanigaVranaPracna}, where a class of finite symplectic polar spaces
and certain finite generalized polygons were successfully linked with
\emph{quantum information theory} (commuting and non-commuting elements of
\emph{Pauli groups}) and the theory of \emph{black holes and black strings}
(symmetry properties of \emph{entropy formulae}). The cited papers contain a
wealth of further references on related work. Another link to quantum
information theory was pointed out in \cite[Section~5.5]{Glynn et al}: The
ambient space $\operatorname*{PG}(7,2)$ of the Segre variety $\mathcal{S}%
_{1,1,1}(2)$ may be regarded as a finite analogue to the \emph{state-space of
a three qubit system} (over the complex numbers).

\section{The Segre variety $\mathcal{S}:=\mathcal{S}_{1,1,1}(2)$
\label{Sec Segre variety S}}

If $V_{2}=V(2,2)$ is a 2-dimensional vector space  over $\operatorname*{GF}%
(2)$ then the tensor product space $V_{8}:=V_{2}\otimes V_{2}\otimes V_{2}%
\ $is of dimension $8$ over $\operatorname*{GF}(2)$ and contains
$3\times3\times3=27$ nonzero decomposable tensors. Viewed as elements of
$\mathbb{P}V_{8}=\operatorname*{PG}(7,2)$ these decomposable tensors
constitute the Segre variety $\mathcal{S}:=\mathcal{S}_{1,1,1}(2)\subset
\operatorname*{PG}(7,2).$ If the three points of the projective line
$\mathbb{P}V_{2}$ are denoted $\{u_{0},u_{1},u_{2}\},$ then
\begin{equation}
\mathcal{S}:=\mathcal{S}_{1,1,1}(2)=\{E_{ijk},\quad i,j,k\in\{0,1,2\}\}\quad
\text{where }E_{ijk}:=u_{i}\otimes u_{j}\otimes u_{k}.\label{Segre 111}%
\end{equation}
The stabilizer $\mathcal{G}_{\mathcal{S}}:=\mathcal{G}_{\mathcal{S}%
_{1,1,1}(2)}<\operatorname*{GL}(V_{8})$ of $\mathcal{S}$ has the semidirect
product structure%
\begin{equation}
\mathcal{G}_{\mathcal{S}}=\mathcal{N}\rtimes\operatorname*{Sym}(3),\quad
\text{where }\mathcal{N}:=\operatorname*{GL}(V_{2})\times\operatorname*{GL}%
(V_{2})\times\operatorname*{GL}(V_{2}),\label{G_S = N sdp Sym(3)}%
\end{equation}
and so is of order $6^{4}=1296.$ Here the element $(a_{0},a_{1},a_{2}),$
$a_{i}\in\operatorname*{GL}(V_{2}),$ of the normal subgroup $\mathcal{N}$ of
$\mathcal{G}_{\mathcal{S}}$ acts upon $V_{8}$ by the tensor product map
$a_{0}\otimes a_{1}\otimes a_{2}\in\operatorname*{GL}(V_{8}),$ and the action
of $\rho\in\operatorname*{Sym}(3)$ upon $V_{8}=\otimes^{3}V_{2}$ is by the
linear operator (see for example \cite[Section 8.8.1]{ShawVol2})
$\rho_{\text{op}}$ given by $\rho_{\text{op}}(v_{1}\otimes v_{2}\otimes
v_{3})=v_{\rho^{-1}1}\otimes v_{\rho^{-1}2}\otimes v_{\rho^{-1}3}.$ (Thus
$\mathcal{G}_{\mathcal{S}}$ is a wreath product $\mathcal{N}\wr
\operatorname*{Sym}(3).)$

\subsection{Some invariant attributes of the Segre variety $\mathcal{S}%
_{1,1,1}(2).$\label{SSec Invariant attributes}}

As well as its twenty-seven points $E_{ijk}$ in (\ref{Segre 111}), we will
make use of the following $\mathcal{G}_{\mathcal{S}}$-invariant attributes of
the Segre variety $\mathcal{S}_{1,1,1}(2).$

\begin{enumerate}
\item The set of twenty-seven \emph{generators} $\{L_{ij}^{r},~i,j\in
\{0,1,2\},~r\in\{1,2,3\}\}$ of $\mathcal{S}_{1,1,1}(2),$ where the lines
$L_{ij}^{r}$ are defined, for $i,j\in\{0,1,2\},$ by%
\begin{equation}
L_{ij}^{1}=\{E_{kij}\}_{k\in\{0,1,2\}},~~L_{ij}^{2}=\{E_{ikj}\}_{k\in
\{0,1,2\}},~~L_{ij}^{3}=\{E_{ijk}\}_{k\in\{0,1,2\}}. \label{27 generators L}%
\end{equation}

\item The set of nine copies $\{\sigma_{i}^{r},~i\in\{0,1,2\},~r\in
\{1,2,3\}\}$ of a Segre variety $\mathcal{S}_{1,1}(2)$ which are contained in
the Segre variety (\ref{Segre 111}), where%
\begin{equation}
\hspace*{-0.4in}\sigma_{i}^{1}:=\{E_{ijk}\}_{j,k\in\{0,1,2\}},~~\sigma_{i}%
^{2}:=\{E_{jik}\}_{j,k\in\{0,1,2\}},~~\sigma_{i}^{3}:=\{E_{jki}\}_{j,k\in
\{0,1,2\}}. \label{nine S_1,1}%
\end{equation}

\item The set of nine ambient spaces $\{Y_{i}^{r}=\langle\sigma_{i}^{r}%
\rangle,\quad i\in\{0,1,2\},~r\in\{1,2,3\}\}$ of the nine $\mathcal{S}%
_{1,1}(2)$s in (\ref{nine S_1,1}), each $Y_{i}^{r}$ being of course a $3$-flat.

\item The set of twenty-seven $3$-flats
\begin{equation}
\{Z_{ijk}=\langle L_{jk}^{1},L_{ik}^{2},L_{ij}^{3}\rangle,\quad i,j,k\in
\{0,1,2\}\}, \label{27 3-flats Z}%
\end{equation}
with $Z_{ijk}=Z(p)$ being spanned by the three generators which pass through
the point $p=E_{ijk}.$

\item The set of twenty-seven \emph{distinguished tangents }$\{L(p),p\in
\mathcal{S}\},$ defined by
\begin{equation}
L(p):=\{p,p^{\prime},p^{\prime\prime}\} \label{L(p)}%
\end{equation}
where $p^{\prime},p^{\prime\prime}$ are those two points of the $3$-flat
$Z(p),$ see (\ref{27 3-flats Z}), which are external to each of the three
planes which are spanned by two of the generators through $p.$
\end{enumerate}

\begin{remark}
It may well help to visualize these attributes of the Segre variety
$\mathcal{S}_{1,1,1}(2)$ by appealing to the `$27$-cube' in \cite[Section
3.1]{GreenSaniga} or \cite[Fig.~1]{HavlicekOdehnalSaniga}. In particular the
nine $\mathcal{S}_{1,1}(2)$s in (\ref{nine S_1,1}) can be visualized as nine
`sections' of the $27$-cube, three horizontal and six vertical.
\end{remark}

\subsection{Some subgroups of $\mathcal{G}_{\mathcal{S}}$
\label{SSec Subgroups of G_S}}

In the following we will often choose $\mathcal{B=}\{e_{1},e_{2},\ldots
,e_{8}\}$ as basis for $V_{8},$ where%
\begin{align}
e_{1}  &  =E_{000},~e_{2}=E_{100},~e_{3}=E_{110},~e_{4}=E_{010},\nonumber\\
e_{8}  &  =E_{111},~e_{7}=E_{011},~e_{6}=E_{001},~e_{5}=E_{101}.
\label{Basis e_i}%
\end{align}
(So the multi-indices of $e_{1},e_{3},e_{5},e_{7}$ have even parity and those
of their $`$opposites' $e_{8},e_{6},e_{4},e_{2}$ have odd parity.) If we view
these basis elements in terms of the eight vertices of the `$8$-cube' in
Figure \ref{fig:cube}
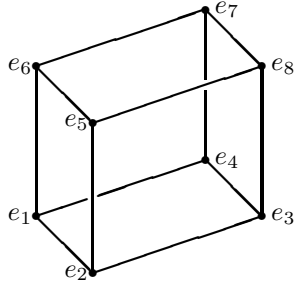
\begin{figure}[h]
\unitlength0.5cm
\centering\begin{picture}(6,7)\thicklines
\put(0,1.5){\line(1,-1){1.5}}
\put(0,1.5){\line(3,1){1.35}}\put(4.5,3){\line(-3,-1){2.85}}
\put(4.5,3){\line(1,-1){1.5}}
\put(1.5,0){\line(3,1){4.5}}
\put(0,1.5){\circle*{0.2}}
\put(4.5,3){\circle*{0.2}}
\put(6,1.5){\circle*{0.2}}
\put(1.5,0){\circle*{0.2}}
\put(0,5.5){\line(1,-1){1.5}}
\put(0,5.5){\line(3,1){4.5}}
\put(4.5,7){\line(1,-1){1.5}}
\put(1.5,4){\line(3,1){4.5}}
\put(0,5.5){\circle*{0.2}}
\put(4.5,7){\circle*{0.2}}
\put(6,5.5){\circle*{0.2}}
\put(1.5,4){\circle*{0.2}}
\put(0,1.5){\line(0,1){4}}
\put(4.5,3){\line(0,1){1.85}}\put(4.5,7){\line(0,-1){1.85}}
\put(6,1.5){\line(0,1){4}}
\put(1.5,0){\line(0,1){4}}
\put(-0.75,1.4){$e_1$}\put(-0.75,5.4){$e_6$}
\put(4.75,2.9){$e_4$}\put(4.75,6.9){$e_7$}
\put(6.25,1.4){$e_3$}\put(6.25,5.4){$e_8$}
\put(0.75,-0.1){$e_2$}\put(0.75,3.9){$e_5$}
\end{picture}
\caption{the `8-cube'}%
\label{fig:cube}%
\end{figure}
then the two regular tetrahedra embedded in this 8-cube have vertices
$\{e_{1},e_{3},e_{5},e_{7}\}$ and $\{e_{2},e_{4},e_{6},e_{8}\}.$ At times we
will view $\langle e_{1},e_{2}\rangle$ as the `$x$-axis', $\langle e_{1}%
,e_{4}\rangle$ as the `$y$-axis', $\langle e_{1},e_{6}\rangle$ as the `$z$-axis'.

Clearly the group $\mathcal{G}_{\mathcal{B}}\cong\operatorname*{Sym}(4)\times
Z_{2}$ of 3-dimensional symmetries of the cube with vertices $\mathcal{B}$
will be a subgroup of $\mathcal{G}_{\mathcal{S}}$. The subgroup $\mathcal{G}%
_{\mathcal{B}}<\mathcal{G}_{\mathcal{S}}$ contains all $4!=24$ permutations of
the four space diagonals $\{e_{1},e_{8}\},\{e_{2},e_{7}\},\{e_{3}%
,e_{6}\},\{e_{4},e_{5}\}$ of the cube and also the central involution
\begin{equation}
J:(18)(27)(36)(45) \label{J}%
\end{equation}
which fixes each space diagonal. (Here, and below, we often use shorthand
notation; in particular $(18)$ is shorthand for $e_{1}\rightleftarrows
e_{8}.)$ Alternatively $\mathcal{G}_{\mathcal{B}}$ may be defined to be that
subgroup of $\mathcal{G}_{\mathcal{S}}$ which fixes the unit point $u:=%
{\textstyle\sum\nolimits_{i=1}^{8}}
e_{i}$ of the basis $\mathcal{B}.$

If $j_{x}\in\operatorname*{GL}(V_{2})$ is the element of order 2 which effects
the interchange $u_{0}\rightleftarrows u_{1},$ then, in shorthand notation,
the element $J_{x}=j_{x}\otimes I\otimes I\in\mathcal{N}\vartriangleleft
\mathcal{G}_{\mathcal{S}}$ is given in terms of the basis (\ref{Basis e_i})
by:%
\begin{equation}
J_{x}:(12)(34)(56)(78). \label{Jx}%
\end{equation}
Similarly we arrive at two further involutions $J_{y},J_{z}\in\mathcal{N}%
\vartriangleleft\mathcal{G}_{\mathcal{S}}$:%
\begin{equation}
J_{y}:(14)(23)(58)(67),\quad J_{z}:(16)(25)(38)(47). \label{Jy,Jz}%
\end{equation}
The three involutions (\ref{Jx}), (\ref{Jy,Jz}) satisfy $J_{x}J_{y}J_{z}=J,$
and are elements of $\mathcal{G}_{\mathcal{B}}.$

Next, if $a\in\operatorname*{GL}(V_{2})$ is the element of order 3 which
effects the cyclic permutation $(u_{0}u_{1}u_{2})$ then the element
$A_{x}=a\otimes I\otimes I\in\mathcal{G}_{\mathcal{S}}$ is given by%
\begin{align}
A_{x}: &  e_{1}\mapsto e_{2}\mapsto e_{1}+e_{2},\quad e_{4}\mapsto
e_{3}\mapsto e_{3}+e_{4},\nonumber\\
&  e_{6}\mapsto e_{5}\mapsto e_{5}+e_{6},\quad e_{7}\mapsto e_{8}\mapsto
e_{7}+e_{8}.\label{A_x}%
\end{align}
Similarly there are two further elements $A_{y}=I\otimes a\otimes I$ and
$A_{z}=I\otimes I\otimes a$, each of order 3, which belong to $\mathcal{G}%
_{\mathcal{S}}$ but not to $\mathcal{G}_{\mathcal{B}}.$

The linear mapping $K_{12}:=\rho_{\text{op}}$ arising from the interchange
$\rho=(12)\in\operatorname*{Sym}(3)$ is another involution $\in\mathcal{G}%
_{\mathcal{B}}$ given by
\begin{equation}
K_{12}:(24)(57),\quad1,3,6,8\text{ fixed.} \label{J12}%
\end{equation}
Other involutions $K_{13},K_{23}\in\mathcal{G}_{\mathcal{B}}$ can similarly be
arrived at. Observe that the product $C:=J_{x}K_{12}$ is an element of
$\mathcal{G}_{\mathcal{B}}$ of order 4:
\begin{equation}
C:(1234)(8765). \label{JxJ12}%
\end{equation}
Setting $B:=\rho_{\text{op}}$ with the choice $\rho=(123)\in
\operatorname*{Sym}(3),$ then $B$ is an element of $\mathcal{G}_{\mathcal{B}}$
of order 3 given by
\begin{equation}
B:(375)(246),\quad1,8\text{ fixed.} \label{B}%
\end{equation}

\subsubsection{The subgroup $\mathcal{G}_{\mathcal{S}}^{0}=\mathcal{N}%
_{0}\rtimes\operatorname*{Sym}(3)$ of $\mathcal{G}_{\mathcal{S}}$
\label{SSSec Subgroup G^0}}

The group $\mathcal{G}_{\mathcal{S}}=\mathcal{N}\rtimes\operatorname*{Sym}(3)$
contains an obvious subgroup of index 2, namely $\mathcal{N}\rtimes
\operatorname{Alt}(3).$ Of relevance to our later concerns is the fact that
$\mathcal{G}_{\mathcal{S}}$ also contains a much less obvious subgroup
$\mathcal{G}_{\mathcal{S}}^{0}$ of index 2, which arises from the special
nature of the group $\operatorname*{GL}(V_{2})$ when the vector space $V_{2}$
is over the field $\operatorname*{GF}(2).$ For then $\operatorname*{GL}%
(V_{2})$ contains a subgroup $\mathcal{H}\cong Z_{3}$ of index 2. It follows
that the group $\mathcal{N}:=\operatorname*{GL}(V_{2})\times\operatorname*{GL}%
(V_{2})\times\operatorname*{GL}(V_{2})$ contains a subgroup $\mathcal{N}_{0}$
of index 2 consisting of those elements $a_{1}\otimes a_{2}\otimes a_{3}%
\in\mathcal{N}$ such that an even number, 0 or 2, of the $a_{i}$ belong to the
coset $\mathcal{K}:=\operatorname*{GL}(V_{2})\setminus\mathcal{H}$ (which
incidentally consists of three involutions). Consequently the subgroup
\begin{equation}
\mathcal{G}_{\mathcal{S}}^{0}:=\mathcal{N}_{0}\rtimes\operatorname*{Sym}(3)
\label{SubgroupG^0}%
\end{equation}
has index 2 in $\mathcal{G}_{\mathcal{S}}$. Observe that the three involutions%
\begin{equation}
J_{x}J_{y}:(13)(24)(57)(68),\quad J_{x}J_{z}:(15)(26)(37)(48),\quad J_{y}%
J_{z}:(17)(28)(35)(46) \label{three involutions JxJy etc}%
\end{equation}
are consequently elements of $\mathcal{G}_{\mathcal{S}}^{0},$ while
$J=J_{x}J_{y}J_{z}\in\mathcal{G}_{\mathcal{S}}\setminus\mathcal{G}%
_{\mathcal{S}}^{0}.$

\begin{remark}
If instead we deal, for general $m>1,$ with $V_{2^{m}}:=\otimes^{m}V_{2},$ and
with $\mathcal{S}=\mathcal{S}_{m}(2):=\mathcal{S}_{1,1,\ldots,1}%
(2)\subset\operatorname*{PG}(2^{m}-1,2),$ then we have $\mathcal{G}%
_{\mathcal{S}}=\mathcal{N}\rtimes\operatorname*{Sym}(m)$ where $\mathcal{N}%
:=\times^{m}\operatorname*{GL}(V_{2})$ has a similarly defined subgroup
$\mathcal{N}_{0}$ giving rise to a subgroup $\mathcal{G}_{\mathcal{S}}%
^{0}:=\mathcal{N}_{0}\rtimes\operatorname*{Sym}(m)$ of $\mathcal{G}%
_{\mathcal{S}}$ of index 2. In the rather special case $m=2$ the nine points
of $\mathcal{S}:=\mathcal{S}_{1,1}(2)$ are the points of a hyperbolic quadric
$\mathcal{H}_{3}$ in $\operatorname*{PG}(3,2)$ and $\mathcal{G}_{\mathcal{S}%
}^{0}$ is seen to consist of those orthogonal transformations which fix
separately each of the two external lines of the quadric.
\end{remark}

\subsubsection{Generators for the groups $\mathcal{G}_{\mathcal{S}%
},~\mathcal{G}_{\mathcal{S}}^{0},~\mathcal{G}_{\mathcal{B}}$
\label{SSSec Generators for groups ...}}

It is possible to generate the group $\mathcal{G}_{\mathcal{S}}$ using just
two elements $M,N:$%
\begin{equation}
\mathcal{G}_{\mathcal{S}}=\langle M,N\rangle. \label{<M,N>}%
\end{equation}
One choice of generators uses the elements $M,\ N,$ each of order 6, defined
by:%
\begin{equation}
M=J_{x}B,\quad\quad N=A_{x}K_{12}. \label{M =  , N =}%
\end{equation}
This choice was checked by use of the computer algebra system Magma, see
\cite{MAGMA}. Magma was also used to check the following related choices of
generators for the groups $\mathcal{G}_{\mathcal{S}}^{0}$ and $\mathcal{G}%
_{\mathcal{B}}$:%
\begin{align}
\mathcal{G}_{\mathcal{S}}^{0}  &  =\langle M^{\prime},N\rangle,\quad
\text{where }M^{\prime}=JM=JJ_{x}B\quad\text{and }N=A_{x}K_{12},\label{M'N}\\
\mathcal{G}_{\mathcal{B}}  &  =\langle M,K_{12}\rangle,\quad\text{where
}M=J_{x}B. \label{MJ12}%
\end{align}

\section{The distinguished $Z_{3}$-subgroup $\mathcal{Z}<\operatorname*{GL}%
(8,2)$ \label{Sec Distinguished subgroup Z}}

An intriguing aspect of the subgroup $\mathcal{G}_{\mathcal{S}}^{0}%
<\operatorname*{GL}(8,2),$ see (\ref{SubgroupG^0}), \emph{is that its
centralizer }$\mathcal{Z}:=C_{\operatorname*{GL}(8,2)}(\mathcal{G}%
_{\mathcal{S}}^{0})$\emph{ in }$\operatorname*{GL}(8,2)$\emph{ is non-trivial.
}To see this, suppose $W\in\mathcal{Z},$ so that
\begin{equation}
WA=AW\quad\text{for all }A\in\mathcal{G}_{\mathcal{S}}^{0}. \label{AW = WA}%
\end{equation}
For each $A\in\mathcal{G}_{\mathcal{S}}^{0}$ it follows from this that $W$
leaves invariant the subspace
\begin{equation}
\operatorname{Fix}(A)=\{x\in V_{8}|~Ax=x\} \label{Fix(A)}%
\end{equation}
of $V_{8}$ consisting of all vectors fixed by $A.$ Observe that for the first
two of the involutions (\ref{three involutions JxJy etc}) we have
\begin{equation}
\operatorname{Fix}(J_{x}J_{y})=\langle13,24,57,68\rangle,\quad
\operatorname{Fix}(J_{x}J_{z})=\langle15,26,37,48\rangle.
\label{three involutions Fix}%
\end{equation}
(Here, and elsewhere, we use $i,ij,ijk,\ldots$ as shorthand for $e_{i}%
,e_{i}+e_{j},e_{i}+e_{j}+e_{k},\ldots\;.)$ Noting that $\operatorname{Fix}%
(J_{x}J_{y})\cap\operatorname{Fix}(J_{x}J_{z})=\langle1357,2468\rangle$ it
follows that if $W\in\mathcal{Z}$ then $W$ necessarily stabilizes the
distinguished tangent $L(u)=\{u,1357,2468\}$, see (\ref{L(p)}), where
$u:=12345678$ denotes the unit point of the basis $\mathcal{B}.$ (Using the
third involution $J_{y}J_{z}$ in (\ref{three involutions JxJy etc}) gives
nothing further, since $\operatorname{Fix}(J_{y}J_{z})=\langle
17,28,35,46\rangle$ meets each of the 4-spaces (\ref{three involutions Fix})
also in the 2-space $\langle1357,2468\rangle.)$

Now $\mathcal{G}_{\mathcal{S}}^{0}$ acts transitively on the 27 points of
$\mathcal{S},$ and so it acts transitively on the 27 distinguished tangents
$\{L(p),p\in\mathcal{S}\}.$ Consequently if $W\in\mathcal{Z}$ then $W$
stabilizes each distinguished tangent. Moreover if $W\in\mathcal{Z}$ fixes a
point on one of the distinguished tangents, then it fixes a point on each of
the distinguished tangents, which can be the case only if $W=I.$ \emph{So if
an element }$W\neq I$ \emph{exists in }$\mathcal{Z}$\emph{ then }$W$ \emph{is
of order 3 and cyclically permutes the three points of each distinguished
tangent.}

Consider the eight distinguished tangents $L_{i}:=L(e_{i})$ through the basis
vectors (\ref{Basis e_i}):%
\begin{align}
L_{1}  &  =\{1,246,1246\},\qquad L_{8}=\{8,8357,357\},\nonumber\\
L_{3}  &  =\{3,248,3248\},\qquad L_{6}=\{6,6157,157\},\nonumber\\
L_{5}  &  =\{5,268,5268\},\qquad L_{4}=\{4,4137,137\},\nonumber\\
L_{7}  &  =\{7,468,7468\},\qquad L_{2}=\{2,2135,135\}.
\label{8 distinguished tangents Li}%
\end{align}
Suppose $W$ is that element of $\operatorname*{GL}(8,2)$ which fixes each of
these eight lines by cycling through their points in the order displayed in
(\ref{8 distinguished tangents Li}); in particular $W^{3}=I$. So $W$ is given
by its effect on the basis $\mathcal{B}$ by%
\begin{align}
W:  &  1\mapsto246,\quad2\mapsto2135,\quad3\mapsto248,\quad4\mapsto
4137,\nonumber\\
&  5\mapsto268,\quad6\mapsto6157,\quad7\mapsto468,\quad8\mapsto8357. \label{W}%
\end{align}

\begin{theorem}
\label{Thm Z=<W>}The centralizer in $\operatorname*{GL}(8,2)$ of the subgroup
$\mathcal{G}_{\mathcal{S}}^{0}$ is the subgroup $\mathcal{Z}:=\langle
W\rangle\cong Z_{3},$ where $W,$ defined as in (\ref{W}), fixes each line of
the invariant set $\{L(p),p\in\mathcal{S}\}$. Moreover the orbits of
$\mathcal{Z}$ in $\operatorname*{PG}(7,2)$ constitute a $\mathcal{G}%
_{\mathcal{S}}$-invariant spread $\mathcal{L}_{85}$ of $85$ lines.\emph{ }
\end{theorem}

\begin{proof}
The fact that $W$ centralizes the subgroup $\mathcal{G}_{\mathcal{S}}^{0}$
follows upon checking that $W$ commutes with each of the generators in
(\ref{M'N}). The fact that $\mathcal{Z}$ fixes (even just one of) the eight
distinguished tangents (\ref{8 distinguished tangents Li}) now ensures that
$\mathcal{Z}$ fixes every one of the invariant set $\{L(p),p\in\mathcal{S}\}.$

To see that $C_{\operatorname*{GL}(8,2)}(\mathcal{G}_{\mathcal{S}}^{0})$ is no
larger than $\mathcal{Z}$ suppose that $\mathcal{Z}^{\prime}:=\langle
W^{\prime}\rangle$ is another subgroup $\cong Z_{3}$ which fixes the $27$
distinguished tangents. Consider the cyclic action of $W$ and $W^{\prime}$ on
the points of the eight distinguished tangents
(\ref{8 distinguished tangents Li}), and observe that any four of these
tangents generate $\operatorname*{PG}(7,2)$. So if there exists a subset of
four of the tangents (\ref{8 distinguished tangents Li}) upon which the action
of $W^{\prime}$ agrees with that of $W$ then $W^{\prime}=W.$ But if no such
subset exists then there exists a subset of more than four the tangents
(\ref{8 distinguished tangents Li}) upon which the action of $W^{\prime}$
agrees with that of $W^{2},$ whence $W^{\prime}=W^{2}.$ So in either case it
follows that $\mathcal{Z}^{\prime}=\mathcal{Z}.$

The minimal polynomial of $W$ is $\mu=t^{2}+t+1,$ since $\mu(W)e_{i}=0$ for
each basis vector $e_{i}\mathfrak{.}$ Consequently $W$ is fixed-point-free on
$\operatorname*{PG}(7,2)$ and so the orbits of $\mathcal{Z}$ in
$\operatorname*{PG}(7,2)$ constitute a $\mathcal{G}_{\mathcal{S}}$-invariant
spread $\mathcal{L}_{85}$ of $85$ lines.
\end{proof}

\begin{remark}
Upon observing that $JWJ^{-1}=W^{2}$ holds for the element $J\in
\mathcal{G}_{\mathcal{S}}\setminus\mathcal{G}_{\mathcal{S}}^{0}$ it follows
that
\begin{equation}
A\mathcal{Z}A^{-1}=\mathcal{Z},\mathcal{\quad}\text{for all }A\in
\mathcal{G}_{\mathcal{S}}. \label{normalizes Z}%
\end{equation}

\end{remark}

\section{Orbits and triplets \label{Sec Orbits and triplets}}

\subsection{The five $\mathcal{G}_{S}$-orbits $\mathcal{O}_{1},\mathcal{O}%
_{2},\mathcal{O}_{3},\mathcal{O}_{4},\mathcal{O}_{5}$ of points
\label{SSec point-orbits}}

Although we will arrive at the five orbits in the order $\mathcal{O}%
_{5},\mathcal{O}_{2},\mathcal{O}_{4},\mathcal{O}_{3},\mathcal{O}_{1}$
nevertheless we have adopted the present numbering so as to be in agreement
with that in \cite[Proposition 5]{HavlicekOdehnalSaniga}.

From the structure (\ref{G_S = N sdp Sym(3)}) of $\mathcal{G}_{\mathcal{S}}$
clearly $\mathcal{O}_{5}:=\mathcal{S}_{1,1,1}(2)$ is a single $\mathcal{G}%
_{\mathcal{S}}$-orbit of length 27. Secondly each of the 9 ambient spaces
$Y_{i}^{r}$ contains 6 points external to $\mathcal{S},$ and such points form
an orbit $\mathcal{O}_{2}$ of length $9\times6=54.$ Thirdly each of the 27
$3$-flats $Z_{ijk}$ contains 2 points, on the distinguished tangent
$L(E_{ijk}),$ which are external to $\mathcal{O}_{2}\cup\mathcal{O}_{5},$
giving rise to an orbit $\mathcal{O}_{4}$ of length $27\times2=54.$ Consider
next the lines through a point $x\in\mathcal{S}$ which meet $\mathcal{S}$ in
one or more further points. Three of these lines are generators, whose union
accounts for $7$ points of $\mathcal{S}.$ Now each of the three $\mathcal{S}%
_{1,1}(2)$'s which contain $x$ contribute $4$ bisecants to $\mathcal{S}$
through $x,$ accounting for a further $3\times4=12$ points of $\mathcal{S}.$
The remaining $8$ points of $\mathcal{S}$ thus give rise to $8$ other
bisecants through $x$, and hence to $8$ points external to $\mathcal{S}.$ All
together there are $\frac{1}{2}(27\times8)=108$ `other bisecants', and the
external points on these bisecants form an orbit $\mathcal{O}_{3}$ of length
$108$. As we now show, the remaining $12(=255-27-54-54-108)$ points of
$\operatorname*{PG}(7,2)$ constitute a fifth $\mathcal{G}_{\mathcal{S}}$-orbit
$\mathcal{O}_{1}.$

If $\mathcal{L}_{s}\subset\mathcal{L}_{85}$ is a partial spread formed from
$s$ lines of the complete spread $\mathcal{L}_{85}$ in Theorem \ref{Thm Z=<W>}
then we will denote by $\mathcal{P(L}_{s})$ the set of $3s$ points underlying
the lines of $\mathcal{L}_{s}.$ If $\mathcal{L}_{27}\subset\mathcal{L}_{85}$
denotes the partial spread consisting of the $27$ distinguished tangents then
$\mathcal{P(L}_{27})=\mathcal{O}_{4}\cup\mathcal{O}_{5}.$ Next note from
(\ref{W}) that $W$ maps the point $e_{1}+e_{3}\in\mathcal{O}_{2}$ to the point
$e_{8}+e_{6}\in\mathcal{O}_{2}.$ Hence, in view of theorem \ref{Thm Z=<W>},
$W$ sends every point of $\mathcal{O}_{2}$ to another point of $\mathcal{O}%
_{2},$ So $\mathcal{O}_{2}=\mathcal{P(L}_{18})$ for a partial spread
$\mathcal{L}_{18}\subset\mathcal{L}_{85}.$ Noting also that $W$ maps the point
$e_{1}+e_{8}\in\mathcal{O}_{3}$ to the point $e_{1}+u\in\mathcal{O}_{3},$ it
similarly follows that $\mathcal{O}_{3}=\mathcal{P(L}_{36})$ for a partial
spread $\mathcal{L}_{36}\subset\mathcal{L}_{85}.$ (Alternatively, $W$ mapping
a subset of $\mathcal{O}_{3}$ into $\mathcal{O}_{1}$ would contradict
$\mathcal{O}_{3}$ being a single $\mathcal{G}_{\mathcal{S}}$-orbit.)
Consequently the $12$-set $\mathcal{O}_{1}$ must be of the form $\mathcal{P(L}%
_{4})$ for a partial spread $\mathcal{L}_{4}\subset\mathcal{L}_{85}.$ Such a
$12$-set $\mathcal{O}_{1}$ possesses $(12\times9)/2=54$ bisecants, and so,
granted the lengths $27,54,54,108$ of the preceding orbits, $\mathcal{O}_{1}$
must be a single $\mathcal{G}_{\mathcal{S}}$-orbit. Explicitly we have
$\mathcal{L}_{4}=\{L_{a},L_{b},L_{c},L_{d}\}$ where
\begin{align}
L_{a} &  =\{18u,357u,246u\},\quad L_{b}=\{27u,135u,468u\},\nonumber\\
L_{c} &  =\{36u,157u,248u\},\quad L_{d}=\{45u,137u,268u\},\label{La ... Ld}%
\end{align}
and where $iju$ and $ijku$ are shorthand for $e_{i}+e_{j}+u$ and $e_{i}%
+e_{j}+e_{k}+u.$ So the points external to $\mathcal{O}_{1}$ on the $54$
bisecants are seen to form the orbit $\mathcal{O}_{2}.$ As an instance of the
action of $\mathcal{G}_{S}$ on $\mathcal{O}_{1},$ observe that $C$ in
(\ref{JxJ12}) effects the cyclic permutation $(L_{a}L_{b}L_{c}L_{d}).$ The
next theorem summarizes the foregoing results.

\begin{theorem}
\label{Thm Orbits}Under the action of $\mathcal{G}_{S}$ the points of
$\operatorname*{PG}(7,2)$ fall into five orbits $\mathcal{O}_{5}%
,\mathcal{O}_{2},\mathcal{O}_{4},\mathcal{O}_{3},\mathcal{O}_{1},$ of
respective lengths $27,54,54,108,12,$ and the lines of the invariant spread
$\mathcal{L}_{85}$ fall into four orbits $\mathcal{L}_{27},\mathcal{L}%
_{18},\mathcal{L}_{36},\mathcal{L}_{4}$, where%
\begin{equation}
\mathcal{P(L}_{4})=\mathcal{O}_{1},\quad\mathcal{P(L}_{18})=\mathcal{O}%
_{2},\quad\mathcal{P(L}_{36})=\mathcal{O}_{3},\quad\mathcal{P(L}%
_{27})=\mathcal{O}_{4}\cup\mathcal{O}_{5}. \label{line and point orbits}%
\end{equation}

\end{theorem}

\begin{remark}
\label{Rmk GF(4) excursion}The results in the theorem agree with those of
Havlicek, Odehnal and Saniga, see \cite[Propositions 4, 5]%
{HavlicekOdehnalSaniga}. These authors arrived at their results after an
excursion into $\operatorname*{GF}(4)$ terrain, viewing the three points
$\{u_{0},u_{1},u_{2}\}$ of the projective line $\mathbb{P}V(2,2)$ as the
`real' points of the projective line $\mathbb{P}V(2,4).$ They obtained thereby
a $\mathcal{G}_{\mathcal{S}}$\emph{-}invariant basis for $\operatorname*{PG}%
(7,4)$ consisting of four pairs of `complex conjugate' points, these yielding
the four real lines $\mathcal{L}_{4},$ and hence the orbit $\mathcal{O}_{1}.$
\end{remark}

\subsection{The triplet $\{\mathcal{S},\mathcal{S}^{\prime},\mathcal{S}%
^{\prime\prime}\}$ of Segre varieties $\mathcal{S}_{1,1,1}(2)$
\label{SSec Triplets S,S',S''}}

Each distinguished tangent $L(p),p\in\mathcal{S}:=\mathcal{S}_{1,1,1}%
(2)=\mathcal{O}_{5},$ contains two points of the orbit $\mathcal{O}_{4},$
namely $p^{\prime}=Wp$ and $p^{\prime\prime}=W^{2}p.$ Consequently we have
\begin{equation}
\mathcal{O}_{4}=\mathcal{S}^{\prime}\cup\mathcal{S}^{\prime\prime}%
,\quad\mathcal{S}^{\prime}=W(\mathcal{S}),\quad\mathcal{S}^{\prime\prime
}=W^{2}(\mathcal{S}), \label{theta3 splits}%
\end{equation}
and so \emph{the }$\mathcal{G}_{\mathcal{S}}$\emph{-orbit }$\mathcal{O}_{4}%
$\emph{ splits into two copies of the Segre variety }$\mathcal{S}.$ It is
quite a surprise to find that the study of a single Segre variety
$\mathcal{S}_{1,1,1}(2)$ inevitably leads one to deal with a triplet of such
varieties which share the same $27$ distinguished tangents and hence give rise
to the same distinguished $Z_{3}$-subgroup $\mathcal{Z}<\operatorname*{GL}%
(8,2)\ $studied in section \ref{Sec Distinguished subgroup Z}. Moreover%
\begin{equation}
\mathcal{G}_{\mathcal{S}}^{0}=\mathcal{G}_{\mathcal{S}^{\prime}}%
^{0}=\mathcal{G}_{\mathcal{S}^{\prime\prime}}^{0}=\mathcal{N}_{0}%
\rtimes\operatorname*{Sym}(3), \label{S,S'S'' share G^0}%
\end{equation}
the group $\mathcal{G}_{\mathcal{S}}^{0}$ having the six point-orbits
$\mathcal{O}_{1},\mathcal{O}_{2},\mathcal{O}_{3},\mathcal{S}^{\prime
},\mathcal{S}^{\prime\prime},\mathcal{S}.$ Further $\mathcal{G}_{\mathcal{S}}$
effects the interchange $\mathcal{S}^{\prime}\rightleftarrows\mathcal{S}%
^{\prime\prime},$ $\mathcal{G}_{\mathcal{S}^{\prime}}$ effects the interchange
$\mathcal{S}\rightleftarrows\mathcal{S}^{\prime\prime}$ and $\mathcal{G}%
_{\mathcal{S}^{\prime\prime}}$ effects the interchange $\mathcal{S}%
\rightleftarrows\mathcal{S}^{\prime}.$ In particular the involution
$J\in\mathcal{G}_{\mathcal{S}}$ effects $\mathcal{S}^{\prime}\rightleftarrows
\mathcal{S}^{\prime\prime},$ the involution $J^{\prime}:=WJW^{-1}=W^{2}%
J\in\mathcal{G}_{\mathcal{S}^{\prime}}$ effects $\mathcal{S}\rightleftarrows
\mathcal{S}^{\prime\prime}$ and the involution $J^{\prime\prime}:=W^{2}%
JW^{-2}=WJ\in\mathcal{G}_{\mathcal{S}^{\prime\prime}}$ effects $\mathcal{S}%
\rightleftarrows\mathcal{S}^{\prime}.$

\subsection{The $\mathcal{G}_{\mathcal{B}}$-orbits of points}

Under the action of the subgroup $\mathcal{G}_{\mathcal{B}}$ of $\mathcal{G}%
_{\mathcal{S}},$ the orbits $\mathcal{O}_{i}$ decompose as in Table
1\footnote{The entries in this table are in accord with those reported in
\cite[Table 5.1]{Glynn et al}}.%

\[%
\begin{tabular}
[c]{|c|c|c|l|l|l|}\hline
\multicolumn{6}{|c|}{Table 1}\\\hline
$\mathcal{G}_{\mathcal{S}}$-orbit & $\mathcal{G}_{\mathcal{B}}$-orbit & $w$ &
$|\mathcal{O}_{i,w}|$ & $p_{i,w}$ & mnemonic\\\hline
$\mathcal{O}_{1}$ & $\mathcal{O}_{1,5}$ & 5 & 8 & $135u$ & \\
& $\mathcal{O}_{1,6}$ & 6 & 4 & $18u$ & \\\hline
$\mathcal{O}_{2}$ & $\mathcal{O}_{2,2}$ & 2 & 12 & $13$ & face diagonals\\
& $\mathcal{O}_{2,3}$ & 3 & 24 & $123$ & 2-arcs\\
& $\mathcal{O}_{2,4}$ & 4 & 6 & $1278$ & opposite edges\\
& $\mathcal{O}_{2,6}$ & 6 & 12 & $12u$ & \\\hline
$\mathcal{O}_{3}$ & $\mathcal{O}_{3,2}$ & 2 & 4 & $18$ & SD (space diagonal)\\
& $\mathcal{O}_{3,3}$ & 3 & 24 & $128$ & SD + point\\
& $\mathcal{O}_{3,4}$ & 4 & 24 & $1238$ & SD + 3-arc\\
& $\mathcal{O}_{3,4^{\prime}}$ & 4 & 24 & $1248$ & SD + 2-arc\\
& $\mathcal{O}_{3,5}$ & 5 & 24 & $123u$ & \\
& $\mathcal{O}_{3,7}$ & 7 & 8 & $1u$ & \\\hline
$\mathcal{O}_{4}$ & $\mathcal{O}_{4,3}$ & 3 & 8 & $135$ & \textquotedblleft
tri-diagonal\textquotedblright\\
& $\mathcal{O}_{4,4}$ & 4 & 8 & $1246$ & \textquotedblleft
claw\textquotedblright\\
& $\mathcal{O}_{4,4^{\prime}}$ & 4 & 2 & $1357$ & tetrahedron\\
& $\mathcal{O}_{4,5}$ & 5 & 24 & $178u$ & \\
& $\mathcal{O}_{4,6}$ & 6 & 12 & $13u$ & \\\hline
$\mathcal{O}_{5}$ & $\mathcal{O}_{5,1}$ & 1 & 8 & $1$ & vertices\\
$=\mathcal{S}_{1,1,1}(2)$ & $\mathcal{O}_{5,2}$ & 2 & 12 & $12$ & edge
midpoints\\
& $\mathcal{O}_{5,4}$ & 4 & 6 & $1234$ & face centres\\
& $\mathcal{O}_{5,8}$ & 8 & 1 & $u$ & cube centre\\\hline
\end{tabular}
\ \ \ \ \ \ \ \
\]
This information will be of help when, in Section
\ref{Sec Invariant polynomials}, we consider $\mathcal{G}_{\mathcal{S}}%
$-invariant polynomials. In the table, $p_{i,w}$ denotes a representative
point on the $\mathcal{G}_{\mathcal{B}}$-orbit $\mathcal{O}_{i,w}$ which
consists of those points of $\mathcal{O}_{i}$ having weight $w$ with respect
to the basis $\mathcal{B}.$ (In the entries for $p_{i,w}$ we have used
shorthand notation, as after equation (\ref{La ... Ld}).) The `mnemonic'
entries in the final column refer to the `8-cube' of the vectors of the basis
$\mathcal{B}.$ Observe that the siblings of $\mathcal{S}=\mathcal{O}%
_{5}=\mathcal{O}_{5,1}\cup\mathcal{O}_{5,2}\cup\mathcal{O}_{5,4}%
\cup\mathcal{O}_{5,8}$ are
\begin{align}
\mathcal{S}^{\prime} &  =W(\mathcal{S})=\mathcal{O}_{4,3}^{\text{even}}%
\cup\mathcal{O}_{4,4}^{\text{odd}}\cup\mathcal{O}_{4,5}^{\text{odd}}%
\cup\mathcal{O}_{4,6}^{\text{even}}\cup\mathcal{O}_{4,4^{\prime}}^{\text{odd}%
}\nonumber\\
\mathcal{S}^{\prime\prime} &  =W^{2}(\mathcal{S})=\mathcal{O}_{4,3}%
^{\text{odd}}\cup\mathcal{O}_{4,4}^{\text{even}}\cup\mathcal{O}_{4,5}%
^{\text{even}}\cup\mathcal{O}_{4,6}^{\text{odd}}\cup\mathcal{O}_{4,4^{\prime}%
}^{\text{even}}.\label{W on S}%
\end{align}
Here a point $p\in\mathcal{O}_{4}$ is classed as `even' if its expression in
terms of the basis $\mathcal{B}$ uses more vectors $e_{i}$ with $i\in
\{2,4,6,8\}$ than with $i\in\{1,3,5,7\};$ otherwise it is classed as `odd'.
Thus if $p=e_{1}+e_{2}\in\mathcal{O}_{5,2}$ then $Wp\in\mathcal{O}%
_{4,5}^{\text{odd}}$ since $Wp=e_{1}+e_{3}+e_{4}+e_{5}+e_{6}.$

\section{Invariant polynomials \label{Sec Invariant polynomials}}

\subsection{The $\mathcal{G}_{\mathcal{B}}$-invariant polynomials of degree
$\leq4$ \label{SSec G_B InvtPoly deg <5}}

The permutational action of $\mathcal{G}_{\mathcal{B}}$ on the eight vectors
$e_{i}\in\mathcal{B}$ gives rise to a corresponding permutational action of
$\mathcal{G}_{\mathcal{B}}$ on the eight coordinates $x_{i}.$ Consequently, by
appeal to Table 1, the $\mathcal{G}_{\mathcal{B}}$-invariant polynomials
$P(x_{1},x_{2},\ldots,x_{8})$ of homogeneous degree $d\leq4$ are as follows.
For example the three $\mathcal{G}_{\mathcal{B}}$-invariant polynomials
$P_{3},P_{3}^{\prime},P_{3}^{\prime\prime}$ of degree 3 arise from the
coordinate orbits corresponding to the three orbits $\mathcal{O}%
_{2,3},\mathcal{O}_{4,3},\mathcal{O}_{3,3}\ $of points in Table 1 which have
weight 3. (\emph{Note: }when dealing with polynomials, we use
$i,ij,ijk,ijkl,\ldots$ as shorthand for $x_{i},x_{i}x_{j},x_{i}x_{j}%
x_{k},x_{i}x_{j}x_{k}x_{l},\ldots$ .)

\noindent(i) $d=1$
\begin{equation}
P_{1}:=%
{\textstyle\sum\nolimits_{i=1}^{8}}
x_{i}=1+2+3+4+5+6+7+8. \label{P_1}%
\end{equation}

\noindent(ii) $d=2$
\begin{align}
P_{2}  &  :=12+14+16+23+25+34+38+47+56+58+67+78\nonumber\\
P_{2}^{\prime}  &  :=13+15+17+35+37+57+24+26+28+46+48+68\nonumber\\
P_{2}^{\prime\prime}  &  :=18+27+36+45. \label{P_2}%
\end{align}

\noindent(iii) $d=3$ \quad(Here $\mathfrak{I}$ denotes the index set
$\{1,2,3,4,5,6,7,8\})$
\begin{align}
P_{3}:=  &  123+124+134+234+125+126+156+256\nonumber\\
&  +146+147+167+467+235+238+258+358\nonumber\\
&  +347+348+378+478+567+568+578+678\nonumber\\
P_{3}^{\prime}:=  &  135+137+157+357+246+248+268+468\nonumber\\
P_{3}^{\prime\prime}:=  &
{\textstyle\sum\nolimits_{i\in\mathfrak{I\setminus}\{1,8\}}}
18i+%
{\textstyle\sum\nolimits_{i\in\mathfrak{I\setminus}\{2,7\}}}
27i+%
{\textstyle\sum\nolimits_{i\in\mathfrak{I\setminus}\{3,6\}}}
36i+%
{\textstyle\sum\nolimits_{i\in\mathfrak{I\setminus}\{4,5\}}}
45i. \label{P_ 3}%
\end{align}

\noindent(iv) $d=4$%
\begin{align}
P_{4}=  &  1234+1256+1467+2358+3478+5678\nonumber\\
P_{4}^{\prime}=  &  1278+1368+1458+2367+2457+3456\nonumber\\
P_{4}^{\prime\prime}=  &  1246+1235+1347+1567+2348+2568+3578+4678\nonumber\\
P_{4}^{\prime\prime\prime}=  &  1357+2468\nonumber\\
P_{4}^{\text{iv}}=  &  1238+1258+1348+1478+1568+1678\nonumber\\
&  +18\text{ terms using the other 3 SDs }27,36,45\nonumber\\
P_{4}^{\text{v}}:=  &  1248+1268+1468+1358+1378+1578\nonumber\\
&  +1237+1257+2357+2467+2478+2678\nonumber\\
&  +2346+2368+3468+1356+1367+3567\nonumber\\
&  +1345+1457+3457+2456+2458+4568. \label{P_4}%
\end{align}

Consequently we have proved the following lemma.

\begin{lemma}
\label{Lem Genaral invt poly d<=4}The most general $\mathcal{G}_{\mathcal{B}}%
$-invariant polynomial of degree $\leq4$ is, for some choice of $c_{1}%
,c_{2},\ldots,c_{4}^{\text{\emph{v}}}\in\operatorname*{GF}(2),$ of the form%
\begin{align}
Q  &  =c_{1}P_{1}+c_{2}P_{2}+c_{2}^{\prime}P_{2}^{\prime}+c_{2}^{\prime\prime
}P_{2}^{\prime\prime}+c_{3}P_{3}+c_{3}^{\prime}P_{3}^{\prime}+c_{3}%
^{\prime\prime}P_{3}^{\prime\prime}\nonumber\\
&  +c_{4}P_{4}+c_{4}^{\prime}P_{4}^{\prime}+c_{4}^{\prime\prime}P_{4}%
^{\prime\prime}+c_{4}^{\prime\prime\prime}P_{4}^{\prime\prime\prime}%
+c_{4}^{\text{\emph{iv}}}P_{4}^{\text{\emph{iv}}}+c_{4}^{\text{\emph{v}}}%
P_{4}^{\text{\emph{v}}}. \label{most general Q_4}%
\end{align}

\end{lemma}

\subsection{The $\mathcal{G}_{\mathcal{S}}$-invariant polynomial functions of
degree $<8$ \label{SSec G_S InvtPoly deg <8}}

Recall from Section \ref{Sec Intro} that, in a given coordinate system,
point-sets $\psi$ in $\operatorname*{PG}(7,2)$ are in bijective correspondence
with polynomial functions $Q$ of the form $Q(x)=\sum x_{i_{1}}\cdots x_{i_{k}%
}$, $1\leq i_{1}<\cdots<i_{k}\leq8$. Those $Q$ which are invariant under the
coordinate transformations arising from the elements of $\mathcal{G}%
_{\mathcal{S}}$ thereby correspond to point-sets $\psi_{Q}$ which are fixed
under the action of $\mathcal{G}_{\mathcal{S}}.$ So if $Q$ is $\mathcal{G}%
_{\mathcal{S}}$-invariant the point-set $\psi_{Q}\subset\operatorname{PG}%
(7,2)$ with equation $Q(x)=0$ must be a union of some of the five
$\mathcal{G}_{\mathcal{S}}$-orbits $\mathcal{O}_{i}$. We now seek all
$\mathcal{G}_{\mathcal{S}}$-invariant polynomial functions $Q(x),x\in V_{8},$
having (reduced) degree $\deg Q<8.$ But, as noted in Section \ref{Sec Intro},
if $\deg Q<8$ then $|\psi_{Q}|$ must be odd, and since $\mathcal{O}_{5}$ is
the only orbit whose length is odd it follows that for some subset $\alpha$ of
$\{1,2,3,4\}$ we must have%
\begin{equation}
\psi_{Q}=\mathcal{O}_{5}\cup_{i\in\alpha}\mathcal{O}_{i}.\label{psi uses O5}%
\end{equation}
Consequently --- leaving aside the zero polynomial arising from the choice
$\alpha=\{1,2,3,4\},$ and so $\psi_{Q}=\operatorname{PG}(7,2)$ --- \emph{there
are precisely fifteen }$\mathcal{G}_{\mathcal{S}}$\emph{-invariant polynomial
functions }$Q$\emph{ of degree }$<8.$ In fact, as we now proceed to prove,
\emph{these fifteen }$G_{\mathcal{S}}$\emph{-invariant polynomials consist of
one quadratic, six quartic and eight sextic polynomials.}

\subsubsection{The $\mathcal{G}_{\mathcal{S}}$-invariant hyperbolic quadric
$\mathcal{H}_{7}$ \label{SSec H_7}}

The $\mathcal{G}_{\mathcal{S}}$-invariant tetrad $\mathcal{L}_{4}%
=\{L_{a},L_{b},L_{c},L_{d}\}$ of lines (\ref{La ... Ld}) gives rise to a
$\mathcal{G}_{\mathcal{S}}$-invariant set $\mathcal{U}_{4}=\{U_{h}%
\}_{h\in\{a,b,c,d\}}$ of four 5-flats, where $U_{h}$ denotes the span of the
three lines $\mathcal{L}_{4}\setminus L_{h}.$ If $P_{h}(x)=0$ is the quadratic
(Lemma \ref{Lem poldeg of flat X}) equation of $U_{h}$ then $Q_{2}%
:=P_{a}+P_{b}+P_{c}+P_{d}$ will be a $\mathcal{G}_{\mathcal{S}}$-invariant
polynomial of degree $\leq2,$ and so, from Lemma
\ref{Lem Genaral invt poly d<=4}, of the form
\begin{equation}
Q_{2}=c_{1}P_{1}+c_{2}P_{2}+c_{2}^{\prime}P_{2}^{\prime}+c_{2}^{\prime\prime
}P_{2}^{\prime\prime}. \label{General invariant Q2}%
\end{equation}
Consider the subset $\mathcal{P}_{81}$ consisting of those $3^{4}=81$ points
$p\in\operatorname*{PG}(7,2)$ of the form $p=%
{\textstyle\sum\nolimits_{h}}
p_{h},~0\neq p_{h}\in L_{h},$ which belong to none of the $U_{h},$ its
complement $(\mathcal{P}_{81})^{\text{c}}$ thus being the set of $255-81=174$
points which belong to the union of the $U_{h}.$ By Theorem \ref{Thm Orbits}
we see that either $\mathcal{P}_{81}=\mathcal{O}_{2}\cup\mathcal{O}_{5}$ or
$\mathcal{P}_{81}=\mathcal{O}_{4}\cup\mathcal{O}_{5}.$ Now the point
$e_{1}+e_{3},$ which from Table 1 is on the orbit $\mathcal{O}_{2},$ is from
(\ref{La ... Ld}) seen to lie in $\langle L_{a},L_{c}\rangle.$ Consequently
$\mathcal{P}_{81}=\mathcal{O}_{4}\cup\mathcal{O}_{5}$ and $(\mathcal{P}%
_{81})^{\text{c}}=\mathcal{O}_{1}\cup\mathcal{O}_{2}\cup\mathcal{O}_{3}.$ It
follows that $Q_{2}(p)=4=0$ for all $p\in\mathcal{O}_{4}\cup\mathcal{O}_{5}.$
Further $Q_{2}(p)=2=0$ for all $p\in\mathcal{O}_{2};$ for example, since
$e_{1}+e_{3}\in\langle L_{a},L_{c}\rangle,$ the point $e_{1}+e_{3}%
\in\mathcal{O}_{2}$ lies on two of the $U_{h},$ namely $U_{b}$ and $U_{d}.$ On
the other hand each point $p\in\mathcal{O}_{1}$ belongs to three of the
$U_{h}$ and so satisfies $Q_{2}(p)=1,$ while each point $p\in\mathcal{O}_{3}$
belongs to just one of the $U_{h}$ and so satisfies $Q_{2}(p)=3=1.$ It follows
that%
\begin{equation}
\psi_{Q_{2}}=\mathcal{O}_{2}\cup\mathcal{O}_{4}\cup\mathcal{O}_{5}.
\label{psi_Q 2 = O_2,4,5}%
\end{equation}

In particular, upon consulting Table 1, it follows from
(\ref{psi_Q 2 = O_2,4,5}) that $Q_{2}(e_{1})=0,$ $Q_{2}(e_{1}+e_{2})=0,$ and
$Q_{2}(e_{1}+e_{3})=0,$ whence in (\ref{General invariant Q2}) $c_{1}%
=c_{2}=c_{2}^{\prime}=0.$ Hence $Q_{2}=P_{2}^{\prime\prime}.$ In fact, upon
appeal to a later result --- see theorem \ref{Thm 7 invt polys of deg <5}
below --- in $\operatorname*{PG}(7,2)$ no polynomial of degree 2 other than
$P_{2}^{\prime\prime}$ is $\mathcal{G}_{\mathcal{S}}$-invariant. Consequently
the following theorem holds.

\begin{theorem}
\label{Thm Q_2}In $\operatorname*{PG}(7,2)$ the quadratic polynomial
\begin{equation}
Q_{2}=x_{1}x_{8}+x_{2}x_{7}+x_{3}x_{6}+x_{4}x_{5} \label{Q2  = 18 + ... +45}%
\end{equation}
is the unique polynomial of degree $2$ which is $\mathcal{G}_{\mathcal{S}}%
$-invariant. Consequently the Segre variety $\mathcal{O}_{5}=\mathcal{S}%
:=\mathcal{S}_{1,1,1}(2)$ is a subset of the $\mathcal{G}_{\mathcal{S}}%
$-invariant hyperbolic quadric $\mathcal{H}_{7}=\mathcal{O}_{2}\cup
\mathcal{O}_{4}\cup\mathcal{O}_{5}\subset\operatorname*{PG}(7,2)$ having
equation $Q_{2}(x)=0.$
\end{theorem}

\begin{remark}
\label{Rmk Can get Sp(8,2) in another way}The Segre variety $\mathcal{S}%
_{1,1,1}(2)$ thus determines a particular orthogonal group $\operatorname{GO}%
_{8}^{+}(2)\cong\operatorname{O}_{8}^{+}(2).2$ in $\operatorname*{GL}(8,2),$
and hence determines that particular $\operatorname{Sp}(8,2)$ subgroup of
$\operatorname*{GL}(8,2)$ which leaves invariant the non-degenerate symplectic
form $B(x,y)=Q_{2}(x+y)+Q_{2}(x)+Q_{2}(y).$ In fact this particular
$\mathcal{G}_{\mathcal{S}}$-invariant symplectic form can be arrived at much
more simply (\textit{cf}.\ \cite[Section 2]{HavlicekOdehnalSaniga}). For since
we are working over $\operatorname*{GF}(2)$ the group $\operatorname*{GL}%
(V_{2})$ coincides with $\operatorname{Sp}(V_{2}),$ the space $V_{2}=V(2,2)$
possessing a unique nonzero symplectic form. So the tensor product space
$V_{8}=V_{2}\otimes V_{2}\otimes V_{2}$ thereby inherits a particular
$\operatorname{Sp}(8,2)$ geometry.
\end{remark}

\begin{remark}
The orders of the simple group $\operatorname{O}_{8}^{+}(2)$ and of
$\operatorname{GO}_{8}^{+}(2),$ are%
\begin{align}
|\operatorname{O}_{8}^{+}(2)|  &  =2^{12}.3^{5}.5^{2}%
.7=174,182,400\label{|SimpleOgroup|}\\
|\operatorname{GO}_{8}^{+}(2)|  &  =2^{13}.3^{5}.5^{2}.7=348,364,800
\label{|FullOgroup|}%
\end{align}
--- see \cite[{\emph{p.~}[85]}]{Atlas}, If $K$ denotes the involution
$e_{1}\rightleftarrows e_{8},$ and if $K^{\prime}$ denotes the involution
$e_{1}\rightleftarrows e_{8},~e_{2}\rightleftarrows e_{7},$ we checked, using
Magma \cite{MAGMA}, that
\begin{align}
\langle\mathcal{G}_{\mathcal{S}},K\rangle &  =\operatorname{GO}_{8}%
^{+}(2),\label{generating GO}\\
\langle\mathcal{G}_{\mathcal{S}},K^{\prime}\rangle &  =\operatorname{O}%
_{8}^{+}(2). \label{generating O}%
\end{align}

\end{remark}

\subsubsection{The six $\mathcal{G}_{\mathcal{S}}$-invariant quartics
\label{SSec Six invariant quartics}}

The ambient space $Y_{i}^{r}$ of $\sigma_{i}^{r},$ see (\ref{nine S_1,1}), is
a $3$-flat and so has a quartic equation, say $P_{i}^{r}(x)=0.$ Consider the
polynomial
\begin{equation}
Q_{4}:=%
{\textstyle\sum\nolimits_{i=0}^{2}}
{\textstyle\sum\nolimits_{r=1}^{3}}
P_{i}^{r}. \label{Q4 = sum of 9 quartics}%
\end{equation}
Each point $x\in\mathcal{S}=\mathcal{O}_{5}$ lies on precisely three of the
nine $3$-flats $Y_{i}^{r},$ and so $Q_{4}(x)=6=0.$ Since each of the six
points $x\in Y_{i}^{r}\setminus\sigma_{i}^{r}$ lies in $Y_{j}^{s}$ only for
$s=r$ and $j=i,\ $we also have $Q_{4}(x)=8=0$ for $x\in\mathcal{O}_{2}.$
Further $Q_{4}(x)=9=1$ for $x\notin\mathcal{O}_{2}\cup\mathcal{O}_{5}.$ So the
81-set $\mathcal{O}_{2}\cup\mathcal{O}_{5}$ has equation $Q_{4}(x)=0$ where
$\deg Q_{4}\leq4.$ To prove that $\deg Q_{4}=4$ we may invoke part (ii) of
lemma \ref{Lem pol deg d iff (i), (ii)}. For, using the basis $\mathcal{B},$
the 3-flat $X:=\langle e_{1},e_{2},e_{7},e_{8}\rangle$ is seen to meet
$\mathcal{\psi}_{Q_{4}}=\mathcal{O}_{2}\cup\mathcal{O}_{5}$ in an even number
of points, namely the six points on the two generators $\langle e_{1}%
,e_{2}\rangle$ and $\langle e_{7},e_{8}\rangle$ together with the two points
$e_{1}+e_{7}$ and $e_{2}+e_{8}.$ Consequently we have proved part (i) of the
following theorem, the $\mathcal{G}_{\mathcal{S}}$ invariance following since
$\mathcal{S}$ determines uniquely the nine $\mathcal{S}_{1.1}$ varieties in
(\ref{nine S_1,1}).

\begin{theorem}
\label{Thm Q_4}(i) The 81-set $\mathcal{O}_{2}\cup\mathcal{O}_{5}$ is a
quartic hypersurface $Q_{4}(x)=0$ in $\operatorname*{PG}(7,2)$ which is
invariant under the action of $\mathcal{G}_{\mathcal{S}}.$

(ii) Using the basis (\ref{Basis e_i}) the quartic polynomial $Q_{4}$ has the
explicit form%
\begin{equation}
Q_{4}=P_{2}^{\prime\prime}+P_{3}^{\prime}+P_{4}^{\prime\prime\prime}%
+P_{4}^{\text{v}}. \label{Q4 = P2P3P4P4}%
\end{equation}

\end{theorem}

\begin{proof}
(ii) Since, from part (i), $\deg Q_{4}=4,$ we know that $Q_{4}$ is as in
(\ref{most general Q_4}). To determine the coefficients $c_{1},c_{2}%
,\ldots,c_{4}^{\text{v}}$ in (\ref{most general Q_4}) we simply use
$Q_{4}(x)=0$ for $x\in\mathcal{O}_{2}\cup\mathcal{O}_{5},$ and $Q_{4}(x)=1$
for $x\in\mathcal{O}_{1}\cup\mathcal{O}_{3}\cup\mathcal{O}_{4},$ confining our
attention to those points $x$ in Table 1 having weight $\leq4.$ From
$Q_{4}(e_{1})=0,$ $Q_{4}(e_{1}+e_{2})=0,$ $Q_{4}(e_{1}+e_{3})=0$ and
$Q_{4}(e_{1}+e_{8})=1,$ it follows that $c_{1}=c_{2}=c_{2}^{\prime}=0$ and
$c_{2}^{\prime\prime}=1.$ So $Q_{4}=P_{2}^{\prime\prime}+$ terms of degree
$>2.$ From $Q_{4}(123)=0$, $Q_{4}(135)=1$ and $Q_{4}(128)=1,\ $and since
$P_{2}^{\prime\prime}(123)=0,P_{2}^{\prime\prime}(135)=0$ and $P_{2}%
^{\prime\prime}(128)=1,$ it follows that $c_{3}=0,c_{3}^{\prime}=1$ and
$c_{3}^{\prime\prime}=0.$ So%
\begin{equation}
Q_{4}=P_{2}^{\prime\prime}+P_{3}^{\prime}+c_{4}P_{4}+c_{4}^{\prime}%
P_{4}^{\prime}+c_{4}^{\prime\prime}P_{4}^{\prime\prime}+c_{4}^{\prime
\prime\prime}P_{4}^{\prime\prime\prime}+c_{4}^{\text{iv}}P_{4}^{\text{iv}%
}+c_{4}^{\text{v}}P_{4}^{\text{v}}. \label{P2+P3+deg=4}%
\end{equation}
Next we consider in turn the six points $1234,1278,1246,1357,1238,1248$ in
Table 1 which have weight 4. Each of the points $x=1234\in\mathcal{O}_{5}$ and
$x=1278\in\mathcal{O}_{2}$ satisfies $P_{2}^{\prime\prime}(x)=P_{3}^{\prime
}(x)=0,$ and so, from $Q_{4}(x)=0$ for $x\in\mathcal{O}_{5}\cup\mathcal{O}%
_{2},$ we obtain $c_{4}=c_{4}^{\prime}=0.$ The point $x=1246\in\mathcal{O}%
_{4}$ satisfies $P_{2}^{\prime\prime}(x)=0$ and $P_{3}^{\prime}(x)=1,$ and so
from $Q_{4}(x)=1$ we obtain $c_{4}^{\prime\prime}=0.$ The point $x=1357\in
\mathcal{O}_{4}$ satisfies $P_{2}^{\prime\prime}(x)=0$ and $P_{3}^{\prime
}(x)=4=0,$ and so from $Q_{4}(x)=1$ we obtain $c_{4}^{\prime\prime\prime}=1.$
The point $x=1238\in\mathcal{O}_{3}$ satisfies $P_{2}^{\prime\prime}(x)=1$ and
$P_{3}^{\prime}(x)=0,$ and so from $Q_{4}(x)=1$ we obtain $c_{4}^{\text{iv}%
}=0.$ Finally the point $x=1248\in\mathcal{O}_{3}$ satisfies $P_{2}%
^{\prime\prime}(x)=1$ and $P_{3}^{\prime}(x)=1,$ and so from $Q_{4}(x)=1$ we
obtain $c_{4}^{\text{v}}=1.$
\end{proof}

\begin{remark}
\label{Rmk About Thm Q_4}Knowing from part (i) of the theorem that $Q_{4}$
does not contain terms of degree $>4$ there was no need in the proof of part
(ii) to consider points of weight greater than 4; such points will necessarily
satisfy the conditions $Q_{4}(x)=0$ for $x\in\mathcal{O}_{2}\cup
\mathcal{O}_{5},$ and $Q_{4}(x)=1$ for $x\in\mathcal{O}_{1}\cup\mathcal{O}%
_{3}\cup\mathcal{O}_{4}.$

Of course the quartic polynomial $Q_{4}$ could alternatively have been
obtained directly from (\ref{Q4 = sum of 9 quartics}) by feeding in the
explicit coordinate forms of the nine $P_{i}^{r}.$
\end{remark}

\begin{remark}
\label{Rmk Zijk yield Q_2}The polynomial $Q_{4}$ in
(\ref{Q4 = sum of 9 quartics}) arose from the nine $3$-flats $Y_{i}^{r}.$ If
instead we consider the corresponding polynomial $Q$ arising from the
twenty-seven $3$-flats $Z_{ijk},$ see equation (\ref{27 3-flats Z}), we
quickly find that $\psi_{Q}=\mathcal{O}_{2}\cup\mathcal{O}_{4}\cup
\mathcal{O}_{5}.$ Consequently $Q$ is not a quartic but is in fact the
quadratic $Q_{2}$ of Theorem \ref{Thm Q_2}.
\end{remark}

\begin{theorem}
\label{Thm Q'_4}The 189-set $\mathcal{O}_{3}\cup\mathcal{O}_{4}\cup
\mathcal{O}_{5}$ is a $\mathcal{G}_{\mathcal{S}}$-invariant quartic
hypersurface in $\operatorname*{PG}(7,2)$ with equation $Q_{4}^{\prime}(x)=0$
where%
\begin{equation}
Q_{4}^{\prime}=P_{2}^{\prime}+P_{3}^{\prime}+P_{3}^{\prime\prime}%
+P_{4}^{\prime}. \label{Q'4 = P2P3P3P4}%
\end{equation}

\end{theorem}

\begin{proof}
The $\mathcal{G}_{\mathcal{S}}$-invariant tetrad $\mathcal{L}_{4}%
=\{L_{a},L_{b},L_{c},L_{d}\}$ of lines (\ref{La ... Ld}) gives rise to a
$\mathcal{G}_{\mathcal{S}}$-invariant set $\{U_{ab},U_{ac},U_{ad}%
,U_{bc},U_{bd},U_{cd}\}$ of six 3-flats, where $U_{hk}:=\langle L_{h}%
,L_{k}\rangle.$ If $P_{hk}(x)=0$ is the quartic (Lemma
\ref{Lem poldeg of flat X}) equation of $U_{hk}$ then the sum $Q_{4}^{\prime
}\ $of the six $P_{hk}$ will be a $\mathcal{G}_{\mathcal{S}}$-invariant
polynomial of degree $\leq4.$ Recalling that the points external to
$\mathcal{O}_{1}=\mathcal{P}(\mathcal{L}_{4})$ on the bisecants of
$\mathcal{O}_{1}$ form the orbit $\mathcal{O}_{2},$ we see that $Q_{4}%
^{\prime}(x)=6=0$ for $x\in\mathcal{O}_{3}\cup\mathcal{O}_{4}\cup
\mathcal{O}_{5},$ while $Q_{4}^{\prime}(x)=3=1$ if $x\in\mathcal{O}_{1}$ and
$Q_{4}^{\prime}(x)=5=1$ if $x\in\mathcal{O}_{2}.$ So $\psi_{Q_{4}^{\prime}%
}=\mathcal{O}_{3}\cup\mathcal{O}_{4}\cup\mathcal{O}_{5}$ as claimed.

Proceeding now on exactly the same lines as in the proof of part (ii) of the
preceding theorem, we quickly arrive at the explicit form
(\ref{Q'4 = P2P3P3P4}) for $Q_{4}^{\prime}$ (showing in particular that
$Q_{4}^{\prime}$ indeed has degree 4).
\end{proof}

\begin{theorem}
\label{Thm 7 invt polys of deg <5}There exist precisely seven $\mathcal{G}%
_{\mathcal{S}}$-invariant polynomials $Q$ of degree $\leq4$ in the coordinates
$x_{1},x_{2},\ldots,x_{8},$ as displayed in the following table:%
\begin{equation}%
\begin{tabular}
[c]{cc|ccccc|cc}\cline{3-7}
&  & \multicolumn{5}{|c|}{$Q(p)$ if $p\in$} &  & \\\cline{1-2}\cline{8-9}%
\multicolumn{1}{|c}{$Q$} & \multicolumn{1}{|c|}{$\deg Q$} & $\mathcal{O}_{1}$
& $\mathcal{O}_{2}$ & $\mathcal{O}_{3}$ & $\mathcal{O}_{4}$ & $\mathcal{O}%
_{5}$ & $\psi_{Q}$ & \multicolumn{1}{|c|}{$|\psi_{Q}|$}\\\hline
\multicolumn{1}{|c}{$Q_{2}$} & \multicolumn{1}{|c|}{2} & 1 & 0 & 1 & 0 & 0 &
$\mathcal{O}_{2}\cup\mathcal{O}_{4}\cup\mathcal{O}_{5}$ &
\multicolumn{1}{|c|}{135}\\
\multicolumn{1}{|c}{$Q_{4}$} & \multicolumn{1}{|c|}{4} & 1 & 0 & 1 & 1 & 0 &
$\mathcal{O}_{2}\cup\mathcal{O}_{5}$ & \multicolumn{1}{|c|}{81}\\
\multicolumn{1}{|c}{$Q_{4}^{\prime}$} & \multicolumn{1}{|c|}{4} & 1 & 1 & 0 &
0 & 0 & $\mathcal{O}_{3}\cup\mathcal{O}_{4}\cup\mathcal{O}_{5}$ &
\multicolumn{1}{|c|}{189}\\
\multicolumn{1}{|c}{$Q_{4}+Q_{4}^{\prime}$} & \multicolumn{1}{|c|}{4} & 0 &
1 & 1 & 1 & 0 & $\mathcal{O}_{1}\cup\mathcal{O}_{5}$ &
\multicolumn{1}{|c|}{39}\\
\multicolumn{1}{|c}{$Q_{2}+Q_{4}$} & \multicolumn{1}{|c|}{4} & 0 & 0 & 0 & 1 &
0 & $\mathcal{O}_{1}\cup\mathcal{O}_{2}\cup\mathcal{O}_{3}\cup\mathcal{O}_{5}$
& \multicolumn{1}{|c|}{201}\\
\multicolumn{1}{|c}{$Q_{2}+Q_{4}^{\prime}$} & \multicolumn{1}{|c|}{4} & 0 &
1 & 1 & 0 & 0 & $\mathcal{O}_{1}\cup\mathcal{O}_{4}\cup\mathcal{O}_{5}$ &
\multicolumn{1}{|c|}{93}\\
\multicolumn{1}{|c}{$Q_{2}+Q_{4}+Q_{4}^{\prime}$} & \multicolumn{1}{|c|}{4} &
1 & 1 & 0 & 1 & 0 & $\mathcal{O}_{3}\cup\mathcal{O}_{5}$ &
\multicolumn{1}{|c|}{135}\\\hline
\end{tabular}
\ \ \label{Table of 7 invariant Q}%
\end{equation}

\end{theorem}

\begin{proof}
We have already met $Q_{2},Q_{4}$ and $Q_{4}^{\prime};$ linear combinations of
these three polynomials yield the further four polynomials displayed in the
last four rows of the table. To prove that there are no $\mathcal{G}%
_{\mathcal{S}}$-invariant polynomials of degree $\leq4$ other than the seven
in the table, recall, see after equation (\ref{psi uses O5}), that there exist
just fifteen $\mathcal{G}_{\mathcal{S}}$-invariant polynomials of degree $<8.$
But looking ahead to Section \ref{SSec Eight invariant sextics}, the remaining
eight invariant polynomials are all of degree 6.
\end{proof}

\subsubsection{The eight $\mathcal{G}_{\mathcal{S}}$-invariant sextics
\label{SSec Eight invariant sextics}}

\begin{theorem}
\label{Thm S_111 has degree 6}The Segre variety $\mathcal{S}_{1,1,1}(2)$ has
polynomial degree $6.$
\end{theorem}

\begin{proof}
From (\ref{Segre 111}) and (\ref{27 generators L}) observe that $\mathcal{S}$
is the union of the nine mutually disjoint lines $L_{ij}^{3},~i,j\in
\{0,1,2\},$ where by lemma \ref{Lem poldeg of flat X} each line $L_{ij}^{3}$
has equation $P_{ij}(x)=0$ where $\deg P_{ij}=6.$ Consider the polynomial
$Q_{6}:=%
{\textstyle\sum\nolimits_{i=0}^{2}}
{\textstyle\sum\nolimits_{j=0}^{2}}
P_{ij}.$ Each point $x\in\mathcal{S}\ $lies on precisely one of the nine lines
$L_{ij}^{3},$ and so $Q_{6}(x)=8=0,$ while if $x\in\operatorname*{PG}(7,2)$ is
exterior to $\mathcal{S}_{1,1,1}(2)$ then it lies on none of the nine lines
$L_{ij}$, and so $Q_{6}(x)=9=1.$ So $\psi_{Q_{6}}=\mathcal{S}=\mathcal{O}%
_{5}.$ (Of course in this proof we could equally well have used instead the
nine lines $L_{ij}^{1}$ or the nine lines $L_{ij}^{2}.)$ From its definition
the polynomial $Q_{6}$ has degree $\leq6.$ To prove that $\deg Q_{6}=6,$
consider the 5-flat $X:=\langle e_{1},e_{3},e_{5},e_{7},e_{2}+e_{4}%
+e_{6}\rangle$ and observe that $X$ meets $\mathcal{S}$ in an even number of
points, namely the 4 points $\{e_{1},e_{3},e_{5},e_{7}\}.$ Hence from lemma
\ref{Lem pol deg d iff (i), (ii)} it follows that $\deg Q_{6}=6.$
\end{proof}

\smallskip

The fact that $Q_{6}$ has degree 6 can also be shown by an explicit
calculation, as in the proof of the next theorem. For this theorem, in
addition to the polynomials defined in Section \ref{SSec G_B InvtPoly deg <5},
we also need the polynomials $P_{5}$ and $P_{6}$ defined by%
\begin{align}
P_{5}  &  =x_{1}x_{3}x_{5}x_{7}(x_{2}+x_{4}+x_{6}+x_{8})+x_{2}x_{4}x_{6}%
x_{8}(x_{1}+x_{3}+x_{5}+x_{7}),\nonumber\\
P_{6}  &  =x_{1}x_{2}x_{3}x_{6}x_{7}x_{8}+x_{1}x_{2}x_{4}x_{5}x_{7}x_{8}%
+x_{1}x_{3}x_{4}x_{5}x_{6}x_{8}+x_{2}x_{3}x_{4}x_{5}x_{6}x_{7}.
\label{P5=  P6=}%
\end{align}
Observe that the polynomial $P_{5}$ can also be expressed as a product:
$P_{5}=P_{1}P_{4}^{\prime\prime\prime}.$

\begin{theorem}
\label{Thm explicit Q_6}The Segre variety $\mathcal{S}_{1,1,1}(2)$ is a
hypersurface in $\operatorname*{PG}(7,2)$ which has the sextic equation
$Q_{6}(x)=0,$ where%
\begin{equation}
Q_{6}=P_{2}^{\prime}+P_{2}^{\prime\prime}+P_{3}^{\prime\prime}+P_{4}^{\prime
}+P_{4}^{\prime\prime\prime}+P_{4}^{\text{v}}+P_{5}+P_{6}. \label{Q_6 =}%
\end{equation}

\end{theorem}

\begin{proof}
With the aid Magma, see \cite{MAGMA}, these explicit coordinate forms for
$Q_{6}=%
{\textstyle\sum\nolimits_{i=0}^{2}}
{\textstyle\sum\nolimits_{j=0}^{2}}
P_{ij}$ were obtained by use of Lemma \ref{Lem poldeg of flat X}.
\end{proof}

By adding $Q_{6}$ to the seven polynomials in (\ref{Table of 7 invariant Q})
we obtain a further seven invariant polynomials of degree 6. Since we have
previously obtained six invariant polynomials of degree 4 and one of degree 2,
we have therefore obtained the full quota, see after equation
(\ref{psi uses O5}), of fifteen $\mathcal{G}_{\mathcal{S}}$-invariant
polynomials of degree $<8.$

\begin{example}
Consider the sextic polynomial $Q_{6}^{\prime}=Q_{6}+Q_{4}+Q_{4}^{\prime},$
which has the particularly simple form $Q_{6}^{\prime}=P_{5}+P_{6}$. Since
$\psi_{Q_{6}}=\mathcal{O}_{5}$ and $\psi_{Q_{4}+Q_{4}^{\prime}}=\mathcal{O}%
_{1}\cup\mathcal{O}_{5}$ it follows that $\psi_{Q_{6}^{\prime}}=(\mathcal{O}%
_{1})^{c}=\mathcal{O}_{2}\cup\mathcal{O}_{3}\cup\mathcal{O}_{4}\cup
\mathcal{O}_{5}.$
\end{example}

\paragraph{Afterthought.}

Consider the sextic polynomial $Q_{2}Q_{4}^{\prime}.$ Observe from the table
(\ref{Table of 7 invariant Q}) that $(Q_{2}Q_{4}^{\prime})(x)\neq0$ only for
$x\in\mathcal{O}_{1},$ leading to an alternative derivation of the sextic
equation for the Segre variety $\mathcal{S}_{1,1,1}(2)=\mathcal{O}_{5},$
namely in the form%
\[
Q_{2}(x)Q_{4}^{\prime}(x)+Q_{4}(x)+Q_{4}^{\prime}(x)=0,
\]
thus avoiding the computation involved in the previous proof of Theorem
\ref{Thm explicit Q_6}. (However we still need to sort out the $4\times
(12+8+24+6)=200$ terms arising from the product of $Q_{2}(x)$ with
$Q_{4}^{\prime}(x)!)$

\bigskip

\noindent{\small Ron \thinspace Shaw, Centre for Mathematics,}

{\small \noindent University of Hull, Hull HU6 7RX, UK}

{\small r.shaw@hull.ac.uk}

{\small \medskip}

{\small \noindent Neil Gordon, Department of Computer Science, }

{\small \noindent University of Hull, Hull HU6 7RX, UK}

{\small n.a.gordon@hull.ac.uk}

{\small \medskip}

{\small \noindent Hans Havlicek, Institut f\"{u}r Diskrete Mathematik und
Geometrie,}

{\small \noindent Technische Universit\"{a}t, Wiedner Hauptstra\ss e
8--10/104}

{\small \noindent A-1040 Wien, Austria}

{\small havlicek@geometrie.tuwien.ac.at}

\end{document}